\documentclass[a4paper,11pt]{article}% {amsart}%
\usepackage{mathrsfs}
\usepackage{amsfonts}
\usepackage[all]{xy}
\usepackage{amsmath}
\usepackage{arcs}
\usepackage{amscd}
\usepackage{extarrows}
\usepackage{CJK}
\usepackage[pdftex]{graphicx,xcolor}
\usepackage{bm}
\usepackage{amssymb}
\usepackage{amscd,latexsym,amsthm,amsfonts,amssymb,amsmath,amsxtra}
\newtheorem{prop}{Proposition}%[section]
\newtheorem{thm}{Theorem}%[section]
\newtheorem{lemma}{Lemma}%[section]
\newtheorem{rem}{Remark}%[section]
\author{Zhangjie Wang}

\newcommand{\En}{E^{(n)}}

\newcommand{\tdep}{{\tilde\epsilon}}

\newcommand{\Leg}[2]{{\brlr{\frac{#1}{#2}}}}
\newcommand{\ALeg}[2]{{\Sqlr{\frac{#1}{#2}}}}

\newcommand{\arr}{\ar@}
\newcommand{\arro}{\ar@/}

\newcommand{\uset}{\underset}
\newcommand{\oset}{\overset}

\newcommand{\vphi}{{\varphi}}

\newcommand{\ABlr}[1]{{\left|#1 \right|}}

\newcommand{\ABBig}[1]{{\Big| #1\Big|}}

\newcommand{\brlr}[1]{{\left( #1\right)}}
\newcommand{\brbig}[1]{{\big(#1 \big)}}
\newcommand{\brBig}[1]{{\Big(#1 \Big)}}

\newcommand{\Brbig}[1]{{\big\{ #1\big\}}}
\newcommand{\BrBig}[1]{{\Big\{ #1\Big\}}}

\newcommand{\Sqlr}[1]{{\left[#1 \right]}}

\font\cyr=wncyr10 scaled \magstep 1 % 为了添加Tate-Shafarevich gp
\newcommand{\Sha}{\mbox{\cyr X}} % Tate-Shafarevich group

 \renewcommand{\Im}{{\mathrm{Im}}}

 \newcommand{\rank}{{\mathrm{rank}}}

\renewcommand{\Re}{{\mathrm{Re}}} 

  \newcommand{\Li}{{\mathrm{Li}}}

\renewcommand{\mod}{\mathrm{mod}}

\newcommand{\half}{\frac{1}{2}}

 \newcommand{\tdP}{{\tilde {P}}}
\newcommand{\tdQ}{{\tilde {Q}}}

 \newcommand{\tdn}{{\tilde {n}}}

 \newcommand{\tdt}{{\tilde {t}}}

\newcommand{\tdeta}{{\tilde {\eta }}}

 \newcommand{\bF}{{\mathbb {F}}}

 \newcommand{\bN}{{\mathbb {N}}}
 
\newcommand{\bQ}{{\mathbb {Q}}}

 \newcommand{\bZ}{{\mathbb {Z}}}

\newcommand{\cA}{{\mathcal {A}}} \newcommand{\cB}{{\mathcal {B}}}

\newcommand{\cO}{{\mathcal {O}}} \newcommand{\cP}{{\mathcal {P}}}
 \newcommand{\cR}{{\mathcal {R}}}
\newcommand{\cS}{{\mathcal {S}}}

\newcommand{\fA}{{\mathfrak {A}}}

\newcommand{\fa}{{\mathfrak {a}}} \newcommand{\fb}{{\mathfrak {b}}}
\newcommand{\fc}{{\mathfrak {c}}}

 \newcommand{\fp}{{\mathfrak {p}}}

\numberwithin{equation}{section}
\title{Congruent Elliptic Curves with Non-trivial Shafarevich-Tate Groups: Distribution Part}

\begin{document}
\maketitle

\begin{abstract}

We study the distribution  of a subclass congruent elliptic curve $\En: y^2=x^3-n^2x$, where  $n$ is congruent to $1\pmod 8$ with all prime factors congruent to $1\pmod 4$. We prove an independence of residue symbol property. Consequently we get the distribution of  rank zero such $\En$ with $2$-primary part of Shafarevich-Tate group  isomorphic to $\brbig{\bZ/2\bZ}^2$.  We also obtain a lower bound of the number of such $\En$ with rank zero and $2$-primary part of Shafarevich-Tate group  isomorphic to $\brbig{\bZ/2\bZ}^{4}$.

\end{abstract}

\begin{CJK}{GBK}{kai}
%%\begin{CJK}{UTF8}{gbsn} %%% Chinese character in different Op.Sys
\maketitle
% \tableofcontents
%\clearpage

\section{Introduction and Main Theorem}

A positive integer $n$ is called a congruent number if it is the area of a right triangle with rational side lengthes; or equivalently, the Mordell-Weil group of the elliptic curve $E^{(n)}: y^2=x^3-n^2x$ has positive rank. Let $E$ be the elliptic curve over $\bQ$ defined by $y^2=x^3-x$, then $E^{(n)}$ is a quadratic twist of $E$.
We are interested in the behavior of arithmetic groups such as Mordell-Weil groups and Shafarevich-Tate groups   in the quadratic twist family of $E$.

Goldfeld  conjectured  that for any elliptic curve over $\bQ$ there are $50\%$ elliptic curves with Mordell-Weil rank $0$ and  $1$ respectively in its quadratic twist family. So far, this conjecture hasn't been verified for any single elliptic curve. The modular curve $X_0(19)$ has genus one and its cusp at $\infty$ is rational over $\bQ$. For the elliptic curve $(X_0(19),[\infty])$,  Vatsal \cite{vatsal1998rank1}
has proved that  there are positive portion rank $0$  elliptic curves in its quadratic twist family, and so is rank $1$.

In this paper, we consider the distribution of congruent elliptic curves  $E^{(n)}$ with Mordell-Weil rank $0$ and non-trivial $2$-primary Shafarevich-Tate groups. For a positive integer $k$, we denote  $Q_k$  to be the set of positive square-free integers $n$ satisfying:
\begin{itemize}
\item $n\equiv 1   \pmod 8$  with exactly $k$ prime factors;
\item any prime factor of $n$ is congruent to $1$ modulo $4$.
\end{itemize}

Our main result in this paper is the following.
\begin{thm}[]\label{mainthm0}
For any positive integer $k$, let $Q_k(x)$ be the set of integers $n\in Q_k$ with $n\le x$, and $P_k(x)$ consist of those  $n\in Q_k(x)$ satisfying
$$\rank_\bZ \En(\bQ)=0, \qquad \Sha(\En/\bQ)[2^\infty]\simeq \brbig{\bZ/2\bZ}^2$$
where $\En(\bQ)$ is the Mordell-Weil group of $\En$ and $\Sha(\En/\bQ)$ is the Shafarevich-Tate group of $\En$. Then
\[\lim_{x\to\infty}\frac{\#P_k(x)}{\#Q_k(x)}=\half \brBig{u_k+ (2^{-1}-2^{-k})u_{k-1}}\]
where $\displaystyle{u_k:=\prod_{i=1}^{\lfloor \frac{k}{2}\rfloor } (1-2^{1-2i})}$ is decreasing to a limit approximate to $0.419$, here $\lfloor \frac{k}{2}\rfloor$ is the maximal integer less or equal to $k/2$.
\end{thm} 

Similar distribution result for the congruent elliptic curves $\En$ with rank $0$ and  $\Sha(\En/\bQ)[2^\infty]\simeq \brbig{\bZ/2\bZ}^4$ is proved in Theorem \ref{mainthm3}. Now we explain the strategy to prove this theorem:
\begin{itemize}
\item By our previous paper  \cite{wzj2015congruent}: $n\in P_k(x)$ can be characterized with $8$-rank of ideal class group of $\bQ(\sqrt{-n})$, then   Jung-Yue \cite{yueJung2011eightrank} reduces this to quartic residue symbol;
\item We count according to $p_i \pmod {16}$ and the residue symbol of prime factors of $n$, which is reduced to count the number of certain matrix over $\bF_2$ by independence of residue symbol Theorem \ref{mainthm2}.
\end{itemize}

To explain the independence of residue symbol property,  we first introduce some notations. For $d\in Q_k$  and $q$ an integer such that $\Leg qp=1 $ for all prime divisor $p$ of $d$, we denote $\Leg{q}d_4$ to  be the quartic residue symbol  defined in  \S2.2. For an odd integer $a$, we define  $\ALeg2{a}=1$ if $a\equiv\pm5\pmod8$ and $0$ else. Let $k$ be a positive integer, we denote $C_k(x)$  to be  all positive square-free integers $n\le x$ with exactly $k$ prime divisors. Let $\alpha=(\alpha_1,\cdots,\alpha_k)$ with $\alpha_l\in \Brbig{ 1, 5, 9, 13 }$ and $\prod_{l=1}^k\alpha_l\equiv 1\pmod 8$, and  $B$ be a $k\times k$ symmetric $\bF_2$-matrix  with rank $k-1$ and every row sum $0$. Then $\fb=\brlr{\ALeg{2}{\alpha_1},\cdots, \ALeg2{\alpha_k}}^T$ lies in the image of $B$ viewed as a linear transform over $\bF_2^k$. Moreover $By=\fb$ has two different
solutions $y,y'\in \bF_2^k$ with $y+y'=(1,\cdots,1)^T$. We assume $z=(z_1,\cdots,z_k)^T$ to be the one of $y,y'$ such that $z_1=1$.  Then we define $C_k(x,\alpha,B)$ to be all $n=p_1\cdots p_k\in C_k(x)$ with $p_1<\cdots<p_k$ satisfying the following conditions:
\begin{itemize}
\item $p_l\equiv \alpha_l  \pmod {16}$ for $ 1\le l\le k$;
\item The Legendre symbol $\Leg{p_l}{p_j}=(-1)^{B_{lj}}$ for all $1\le l<j\le k$;
\item $\brBig{\frac{2d}{n/d}}_4\Leg{2n/d}d_4=(-1)^{\frac{n-1}8+\frac{d-5}4}$  with $d=\prod p_l^{z_l}$.
\end{itemize}

%hence corresponding $d_1, d_2$ satisfies $d_1d_2=n$,  so this will not affect the definition of $C_k(x,\alpha,B)$.

Now we can state the independence  of residue symbol property:
\begin{thm}\label{mainthm2} Let $\alpha=(\alpha_1,\cdots,\alpha_k)$ with $\alpha_l\in \Brbig{ 1, 5, 9, 13 }$ and $\prod_{l=1}^k\alpha_l\equiv 1\pmod 8$. Then for  any $k\times k$ symmetric matrix $B$ over $\bF_2$ with every row sum $0$ and rank $k-1$, we have
\[\#C_k(x,\alpha,B)\sim \frac1{2^{3k+\binom{k}2+1}}\cdot\# C_k(x)\]
where $\binom k2$ is the binomial coefficient and $\sim$ means that the ratio of its two sides has limit $1$ as $x$ goes to infinity. 
\end{thm}

Rhoades \cite{rhoades20092} % implicitly
claimed  a special case of above Theorem. Moreover he proved an independence of residue symbol property with the method of Cremona-Odoni \cite{cremona1989some} over $\bQ$. For Theorem \ref{mainthm2}, we have to extend the method of  Cremona-Odoni to $\bQ(i)$ because of the quartic residue symbol, whence parallel results like the explicit formula for $\psi(x,\chi)$, Siegel Theorem and Page Theorem are needed. Moreover, we have to transit  from primes to prime ideals and  deal with  some  difficulty in counting  certain residue classes, this can be best seen in the case $k=1$ (\S3.1). 

Since we will use analytic number theory, we will use many standard symbols in analytic number theory, such as $\sim, o(\cdot), O(\cdot), \ll, \pi(x), \Li(x), \psi(x)$, they can be find in any book on analytic number theory, for example Iwaniec-Kowalski \cite{iwaniec2004analytic}.

In the end of this introduction, we give the organization of this paper.  We devote Section 2 to give some preliminary results. Concretely, in \S2.1 we summarize the method of Cremona-Odoni in a   simpler case; Since many residue symbols are used, we give their definition and prove some properties in \S2.2; Those parallel analytic number theory results are enumerated in \S2.3. With these preparation, we prove Theorem \ref{mainthm2} in Section 3, especially we split out the case $k=1$ to outstand the main difficulty beside the idea of Cremona-Odoni in \S3.1. The distribution  result is carried out  in Section 4.

\section{Preliminary Section}
\subsection{Basic Idea}\quad

Since the method of Cremona-Odoni \cite{cremona1989some} plays an important role in our proof of independence of residue symbol property. Now we explain their basic idea in a much simpler case:
\[\#C_k(x)\sim\frac{1}{(k-1)!} \cdot \frac x{\log x}\cdot (\log\log x)^{k-1} \]
where $C_k(x)$ denotes all square-free positive integer $n\le x$ with exact $k$ prime factors. For $k=1$ this is prime number theorem.

For $k\ge2$ their key idea is to consider the induction map
\[C_{k}(x) \oset\vphi\longrightarrow C_{k-1}(x),\; n\mapsto n/\tdn\]
where $\tdn$ is the maximal prime divisor of $n$.  Note $t\in C_{k-1}(x)$ is in the image of $\vphi$ if and only if  there is a prime $p$ such that $\tdt<p\le xt^{-1}$, thus we get:
\[\#C_{k}(x)=\sum_{t\in C_{k-1}(x)} \# \Brbig{ p \text{ prime } \big| \; \tdt< p\le xt^{-1}}\]
Then the following Lemma 3.1 of Cremona-Odoni \cite{cremona1989some} implies that  only $t\in (\mu, \nu]\cap C_{k-1}(x) $ contributes to the main term of $\#C_k(x)$ with $\mu=(\log x)^{100} , \nu=\exp\brBig{\frac{\log x}{(\log\log x)^{100} }}$:
\begin{lemma}\label{lem-Cremona}
If either $m=20, n=\mu$, or  $m=\nu, n=x^{\frac {k-1}{k}}$, then we have:
\begin{eqnarray*}
\sum_{m<t\le n}^* \Li(xt^{-1})&=&
 o\brBig{\frac{x \cdot (\log\log x)^{k-1}}{\log x} }\\
\sum_{\mu<t\le \nu}^* \Li(xt^{-1})&\sim&
\frac1{k-1}\cdot \#C_{k-1}(x)\cdot \log\log x
\end{eqnarray*}
Where  $\uset{{a<t\le b} } { \oset*\sum }f(t):= {\uset{t\in (a,b]\cap C_{k-1}(\infty)}\sum } f(t)$.
\end{lemma}
Whence we reduce to estimate $\uset{\mu<t\le\nu}{ \oset*\sum }\pi(xt^{-1})$. From prime number theorem we only need to estimate 
$\sum_{\mu<t\le \nu}^* \Li(xt^{-1})$,  then  Lemma \ref{lem-Cremona} and induction give
\[\#C_k(x)\sim\frac{1}{(k-1)!} \cdot \frac x{\log x} (\log\log x)^{k-1}\]

For the problem considered by Cremona-Odoni \cite{cremona1989some}, they have to use  $\psi(xt^{-1},a,q)$ instead of $\pi(xt^{-1})$. This forces they have to use explicit formula of $\psi(x,\chi)$, which brings the additional difficulty to deal with possible Siegel zeros in  the error term. Due to Page's Theorem, Siegel zeros are so rare that the sum of all the trivial estimation of $\psi(xt^{-1},\chi)$ with possible Siegel zero contributes to an error term.

Comparing with Cremona-Odoni, we have to deal with corresponding multiplicative number theory over $\bQ(i)$.
%For our independence of residue symbol property
% As in classical Dirichlet prime number Theorem,
\subsection{ Residue Symbols}\quad

In this subsection, we will introduce several residue symbols that will be widely used in this paper.

For $\lambda$ a prime in Gaussian integers $\bZ[i]$ coprime with $1+i$ and $\alpha$ a Gaussian integer,  the quartic residue symbol $\Leg{\alpha}{\lambda}_4$   is defined to be  the unique element in $\Brbig{ \pm 1, \pm i, 0 }$  such that
\[ \alpha^{\frac{\lambda\bar\lambda-1}4} \equiv\Leg\alpha\lambda_4 \pmod \lambda \]
holds over $\bZ[i]$, where $\overline\lambda$ is the   conjugate of $\lambda$. The reference is Ireland-Rosen \cite{ireland1982classical}.

For two Gaussian primes $\lambda_1, \lambda_2$,  we easily deduce that
\[\Leg{\lambda_1}{\bar{\lambda_2}}_4 \Leg{\bar{\lambda_1}}{\lambda_2}_4=1\]
Moreover, we have the quartic reciprocity law
\[\Leg{\lambda_1}{\lambda_2}_4=\Leg{\lambda_2}{\lambda_1}_4 (-1)^{\frac{N\lambda_1-1}4\frac{N\lambda_2-1}4}\]
where $N$ denotes the norm from $\bQ(i)$ to $\bQ$.  If $\theta=\prod_{l=1}^k \lambda_l$ with $\lambda_l$ prime  and  coprime with $1+i$, we define\[\Leg\alpha\theta_4=\prod_{l=1}^k\Leg\alpha{\lambda_l}_4\]

We say an integer $\theta\in\bZ[i]$ is primary if  $\theta \equiv 1 \pmod{2+2i}$. Since we will frequently consider $\Leg2\theta_4$,   we compute it in the following lemma:
\begin{lemma}\label{lem-2quartic}
If $\theta=a+2bi$ is a primary integer with $a,b\in \bZ$,  then $\Leg2\theta_4=i^{-b}$.
\end{lemma}
\begin{proof}
If $\theta$ has rational prime factor $p$ congruent to $3 \pmod 4$, then $\Leg2{-p}_4=1$ and there must be even many such factors counted with multiplicity, so their product is congruent to $1\pmod4$. Hence we reduce to show that $\theta$ has only Gaussian prime factors.

Now we induct on the prime factors of $\theta$. If $\theta$ is a prime, then $\Leg2\theta_4=i^{-b}$, see  P53 of Iwaniec-Kowalski \cite{iwaniec2004analytic}.  For $\theta$ has $k\ge2$ prime factors, then $\theta=\theta_1\theta_2$ with $\theta_l$ a primary Gaussian integer having less than $k$  prime factors for $l=1,2$. Hence if we denote $\theta_l=a_l+2b_li$ with $a_l,b_l\in \bZ$, then induction implies that
\[\Leg2{\theta_l}_4=i^{-b_l}\] Hence $\Leg2{\theta}_4=i^{-(b_1+b_2)}$ by definition.
On the other hand,
\[a+2bi=\theta=\theta_1\theta_2=a_1a_2-4b_1b_2+2(a_1b_2+a_2b_1)i\]
Therefore $b=a_1b_2+a_2b_2$.  Since $\theta_l$ is primary, we have $ a_l-2b_l\equiv  1\pmod 4$. Consequently $b\equiv (1+2b_1)b_2+(1+2b_2)b_1\equiv b_1+b_2 \pmod 4$. So by induction, the Lemma is proved.
\end{proof}

For $p$ a prime congruent to $1 \pmod 4$, there are exactly two primitive  primes $\lambda, \bar\lambda$ lying  above $p$ with $p=\lambda\bar\lambda$. For $q$ a rational integer with $\Leg pq=1$, then the two quartic residue symbol $\Leg q\lambda_4=\Leg q{\bar\lambda}_4=\pm1$. Then we use the symbol $\Leg qp_4$ to denote $\Leg q\lambda_4$ as in Jung-Yue \cite{yueJung2011eightrank}. Note this symbol has the convenience that we needn't choose which primary prime above $p$.  Moreover if $d$ is a positive integer with all prime factors congruent to $1 \pmod 4$, and $q$ such that $\Leg qp=1$ for any prime factor $p$ of $d$, then
\[\Leg qd_4:=\prod_{p|d} \Leg qp_4^{v_p(d)}\]
where $v_p(d)$ denotes the $p$-adic valuation of $d$.

For  future application, we introduce general Legendre symbol over $\bZ[i]$ as in Page 196 of Hecke \cite{heckeGTM77}: Let $\fp$ be a  prime ideal   coprime with  $(1+i)$, the general Legendre symbol $\Leg{\alpha}\fp$   is defined to be the unique element of $\Brbig{ \pm1, 0 }$ such that
\[\alpha^{\frac{N\fp-1}2}\equiv \Leg\alpha\fp \pmod \fp\]
holds. If $\lambda$ is the unique primary prime in $\fp$, we also denote
\[\Leg\alpha\lambda=\Leg\alpha\fp\]
If $\theta=\prod_{l=1}^k \lambda_l$ with $\lambda_l$ primary prime, we denote
\[\Leg\alpha\theta=\prod_{l=1}^k\Leg\alpha{\lambda_l} \]

Another residue symbol is needed, for $p$ a rational prime and $a$ a rational integer coprime with $p$, the additive Legendre symbol $\ALeg ap$ is $1$ if the Legendre symbol $\Leg ap=-1$,  else it is $0$. Similarly for $d$ a positive odd integer, we denote $\ALeg ad=1$ if the Jacobi symbol $\Leg ad=-1$ and $\ALeg ad=0$ if $\Leg ad=1$.

\subsection{Analytic results over number fields}\quad

%\subsection{Analytic results over number fields}
%In this subsection, we first recall some analytic results over number fields.

% First we give some notations and  notions.
Let $K$ be a  number field of degree $n$ with  discriminant $\Delta$  and  ring of algebraic integers  $\cO$. A non-zero element $\gamma\in K$ is  totally positive if  it is positive under all real embeddings. If $K$ has no real embedding, then  totally positive  means  non-zero.   For an integral ideal $\dag$ and $\gamma\in K$, the notation $\gamma\equiv 1\pmod \dag$ means that $\gamma\in \cO_\fp$ and $\gamma\equiv 1 \pmod{ \fp^{v_{\fp}(\dag)}}$ if $\fp\mid\dag$, where $\cO_\fp$ is the integer ring of the $\fp$-adic completion of $K$. Let $P_\dag$ be the group of  principal fractional ideals $(\gamma)$  with $\gamma$   totally positive and $\gamma\equiv 1\pmod\dag$, and $I(\dag)$ denote the set of all the fractional ideals that are coprime with $\dag$. We say $\chi$ is a character modulo an ideal $\dag$  if $\chi$ is a character induced from $I(\dag)/P_\dag$. Then  $\psi(x,\chi)$ is defined to be
\[\psi(x,\chi)=\sum_{N_K\fa\le x} \chi(\fa) \Lambda(\fa)\]
where $\fa$ runs over all the integral ideal with norm less than or equal to $x$ and $N_K$ denotes the norm from $K$ to $\bQ$, and $\Lambda(\fa)$ is the Mangoldt function
\begin{equation*}
\begin{cases}
\log N_K\fp& \text{if $\fa=\fp^m$ with $m\ge 1$ },\\
0 & \text{ else}.
\end{cases}
\end{equation*}
and $\chi(\fa)=0$ if $\fa$ is not coprime with $\dag$. For $\psi(x,\chi)$, we have the following   explicit formula ( see Iwaniec-Kowalski \cite{iwaniec2004analytic} P114):
\begin{prop}\label{mainthm-explicitformula}\quad\\
For $\chi$ a non-principal character mod $\dag$ and $1\le T\le x$, then \begin{equation}\label{eq:expl-iwane}
\psi(x,\chi)=-\sum_{|\Im\rho|\le T}\frac{x^\rho-1}\rho+
O\brBig{xT^{-1}\cdot\log x \cdot \log(x^n\cdot N\dag)}
\end{equation}
where $\rho$ runs over all the zeros of $L(s,\chi)$ with $0\le \Re\rho \le 1$ and $|\Im\rho|\le T$, and the implied constant depends only on $K$.
\end{prop}

For further application, we  introduce   Siegel's Theorem   and Page's Theorem over $K$. For Siegel's Theorem over $K$, we refer to  Fogels \cite{Fogels1963}, \cite{Fogels1965}, \cite{Fogels1968}, while for Page's Theorem over $K$  we refer to Hoffstein-Ramakrishnan \cite{hoffstein1995siegel}.
\begin{prop} \label{thm-Fogels}
\begin{enumerate}
\item[(1)]  For $\chi$ a character modulo $\dag$, and $D=|\Delta| N\dag>D_0>1$,
\begin{enumerate}
\item[(i)] there is a positive constant $c$( which only depends on $n$ ) such that in the region
\[\Re (s)> 1-\frac c{\log D(1+\ABlr{\Im (s)})}>\frac34\quad (*)\]
there is no zero of $L(s,\chi)$ with $\chi$ complex, for at most one real $\chi' $ there maybe a simple zero $\beta'$ of $L(s,\chi')$;
\item[(ii)] If $\beta'$ is the exceptional zero of the exceptional character $\chi'$ modulo $\dag$, then for any $\epsilon>0$, there is a positive constant $c(n,\epsilon)$ such that
    \[1-\beta'> c(n,\epsilon) D^{-\epsilon}\]
\end{enumerate}
\item[(2)] For any $z\ge 2$, and $c_0$ is a suitable constant, then  there is at most a real primitive character $\chi$ to a modulus $\dag$ with $N_K\dag\le z$  has a real zero $\beta$ satisfying $$\beta >1-\frac{c_0}{\log z}$$
\end{enumerate}
\end{prop}

For future purpose,  the first term of the formula (\ref{eq:expl-iwane}) have to be estimated. Similar as classical case, we can get the following explicit version of formula (\ref{eq:expl-iwane}). \begin{equation}\label{eq:explicit-we need}
\psi(x,\chi)=-\frac{x^{\beta'}}{\beta'}+
R(x,T)
\end{equation}
with
\[R(x,T)\ll x\cdot\log^2(x\cdot N_K\dag)\cdot\exp\brBig{-\frac{c_1\log x}{\log |T\cdot N_K\dag|}}+xT^{-1}\log x\cdot\log \ABBig{x^n \cdot N_K\dag}+x^{\frac14}\log x\]
The term $-\frac{x^{\beta'}}{\beta'}$ occurs only if $\chi$ is a real character for which has a zero $\beta'$ (then must be unique and simple) with
\[\beta'>1-\frac {c_2}{\log N_K\dag}\]
where $c_2$ is a certain constant.

\section{Independence of residue symbol property}
To prove independence of residue symbol Theorem \ref{mainthm2}, we have to identify the set $C_k(x,\alpha,B)$ to a set counts certain integers over $\bQ(i)$. For this purpose, we introduce some notations.

Denote  $\cP$ to be the set of all primary primes in $\bZ[i]$ with imaginary part positive. Let $k\ge 1$, and $\alpha=(\alpha_1,\cdots,\alpha_k)$ with $\alpha_l\in\Brbig{ 1, 5, 9, 13 }$ and $\prod_{l=1}^k \alpha_l \equiv 1 \pmod 8$. For $B=B_{k\times k}$  a  symmetric $\bF_2$-matrix  with rank $k-1$ and every row sum $0$, we define $C_k'(x,\alpha,B)$   to be all $\eta=\prod_1^k\lambda_l$ satisfying
\begin{itemize}
\item $N\lambda_1<\cdots<N\lambda_k$ with $ \lambda_l\in\cP$ and $N\eta\le x$;
\item $N\lambda_l\equiv \alpha_l  \pmod {16}$ and $\Leg{N\lambda_l}{N\lambda_j}=(-1)^{B_{lj}}$ for all $l<j$;
\item $\Leg{\theta_2}{\theta_1}\Leg2{\eta}_4=
  (-1)^{\frac{\prod_1^k\alpha_j-1}8 + \frac{\prod_1^k\alpha_j^{z_j}-5}4 }$ where $\theta_1\theta_2=\eta$ and $\theta_1=\prod_1^k \lambda_l^{z_l}$.
\end{itemize}
where $z=(z_1,\cdots,z_k)^T\in \bF_2^k$ satisfying $Bz=\brlr{\ALeg2{\alpha_1}, \cdots, \ALeg2{\alpha_k}}^T$ with $z_1=1$.

For $n=p_1\cdots p_k\in C_k(x,\alpha,B)$ with $p_l$ arranged increasingly, and for any $p_l$ we chose the unique $\lambda_l\in \cP$ such that $p_l=\lambda_l\overline{\lambda_l}$. Then we claim that $\eta=\prod_{i=1}^k \lambda_l \in C_k'(x,\alpha,B)$. Note that $\eta$ obviously satisfies the first  and second conditions in defining $C_k'(x,\alpha,B)$. For the third: consider the $l$-th row of both sides of $Bz=\brlr{ \ALeg2{p_1},\cdots,\ALeg2{p_k}}^T$, we get
$$\sum_{j=1,j\not= l}^k z_j B_{lj}+z_l\sum_{j=1,j\not=l}^k B_{lj}=\ALeg2{p_l}, \quad B_{lj}=\ALeg{p_j}{p_l} \text{ if } j\not=l $$
since every row sum of $B$ is $0$. Then we have $\Leg{2n/d}{p_l}=1$ if $z_l=1$ and $\Leg{2d}{p_l}=1$ if $z_l=0$. Thus the notations $\Leg{2d}{n/d}_4, \Leg{2d}{n/d}_4$  are meaningful, according to their definition we have:
\begin{eqnarray*}
\Leg{2n/d}d_4 \Leg{2d}{n/d}_4&=&\Leg{2p_{t+1}\cdots p_k}{\lambda_1\cdots \lambda_t}_4
\Leg{2p_1\cdots p_t}{\lambda_{t+1}\cdots \lambda_k}_4  \\
&=&\Leg2\eta_4 \cdot \prod_{l=1}^t\prod_{j=t+1}^k \Leg{p_j}{\lambda_l}_4\Leg{p_l}{\lambda_j}_4  \\
&=&\Leg2\eta_4 \cdot \prod_{l=1}^t\prod_{j=t+1}^k \Leg{\lambda_j}{\lambda_l}_4\Leg{\overline{\lambda_j}}{\lambda_l}_4
\Leg{\lambda_l}{\lambda_j}_4\Leg{\overline{\lambda_l}}{\lambda_j}_4
\end{eqnarray*}
where we have assumed that $d=p_1\cdots p_{t}$ for simplicity of notation. Using quartic reciprocity law for $\Leg{\lambda_l}{\lambda_j}_4 $ and $\Leg{\overline{\lambda_l}}{\lambda_j}_4$ we get
\[ \Leg{\lambda_j}{\lambda_l}_4\Leg{\overline{\lambda_j}}{\lambda_l}_4
\Leg{\lambda_l}{\lambda_j}_4\Leg{\overline{\lambda_l}}{\lambda_j}_4
=\Leg{\lambda_j}{\lambda_l}  \Leg{\overline{\lambda_j}}{\lambda_l}_4
\Leg{\lambda_j}{\overline{\lambda_l}}_4 \]
From $\Leg{\overline{\lambda_j}}{\lambda_l}_4
\Leg{\lambda_j}{\overline{\lambda_l}}_4=1$ we obtain
\begin{equation}\label{eq:leg2dd'2d'd}
\Leg{2n/d}d_4 \Leg{2d}{n/d}_4=
\Leg2\eta_4 \Leg{\theta_2}{\theta_1}
\end{equation}
Thus $\eta\in C_k'(x,\alpha,B)$. From this  we obtain   a bijection
\begin{equation}\label{bij-Ck}
C_k(x,\alpha,B)\to C_k'(x,\alpha,B)
\end{equation} 
%So we can identify $C_k(x,\alpha,B)$ with $C_k'(x,\alpha,B)$.
\quad\\

Now we divide the proof of Theorem \ref{mainthm2} into two cases according to $k=1$ or not. The reason  is twofold:  First this case will not use the method of Cremona-Odoni, second we can see the main difference  between Cremona-Odoni and our situation, which makes the proof of the case $k\ge2$ more natural and not too long.

\subsection{The case $k=1$}\quad

For the case $k=1$, then we have $\alpha_1\in \Brbig{ 1, 9 }$. Moreover only $B=0_{1\times 1}$ has rank $k-1=0$, so  $C_1'(x,\alpha_1,0)$ consists all primary primes $\lambda\in\cP$ such that:
$$N\lambda\le x,\quad N\lambda\equiv \alpha_1\pmod{16}, \quad \Leg2{\lambda}_4=(-1)^{\frac{\alpha_1-9}8}$$
If we let $A_{16}$ be those primary  classes $a$ of  $\bZ[i]/16\bZ[i]$ such that
\begin{itemize}
\item[(i)] $Na\equiv \alpha_1\pmod {16}$;
\item[(ii)] $\Leg2{a}_4=(-1)^{\frac{\alpha_1-9}8}$.
\end{itemize}
then by Lemma \ref{lem-2quartic} we yield
\begin{equation}\label{k=1:C1'}
\#C_1'(x,\alpha_1,0)=\half\pi'(x,A_{16},16)
\end{equation}
where $\pi'(y,A,\gamma)$ is the number of primes $\lambda$ in $\bZ[i]$ with $N\lambda\le y$ and $\lambda\in A \pmod \gamma$, and the additional factor $\half$ comes from $\lambda\in C_1'(x,\alpha,0)$ with positive imaginary part.

Since Dirichlet prime ideal Theorem over $\bQ(i)$ concerns prime ideals while our estimation concerns  prime elements, we have to bridge this gap by the following:  Let $\fc$   be the  ideal   ${16}\bZ[i]$,
then by Theorem 6.1 of Lang \cite{lang110gtm}  we have the following long exact sequence:
\begin{eqnarray}\label{k=1:exact-seq}
\xymatrix{
1\ar[r]^{} & \bZ[i]^\times \ar[r]^{} &
\brBig{\bZ[i]/\fc}^\times \ar[r]^{f} &
I(\fc)/P_\fc\ar[r]^{} & 1    }
\end{eqnarray}
where $f$ is induced by mapping every $\fc$-invertible Gauss integer $a$ to the class generated by $(a)$. In fact, the exactness of (\ref{k=1:exact-seq}) can be verified directly using $\bZ[i]$ having class number $1$. Furthermore, we introduce the notation  $\pi(y,\fA,\fa)$ to denote all those prime ideal lies in the classes $\fA$  modulo $P_\fa$ with norm less or equal to $y$.

Now we can transit to prime ideals: let $\fA_{16}$ be the image of $A_{16}$ under $f$, then we get
\begin{equation}\label{k=1:pi'}
\pi'(x,A_{16}, {16}) = \pi (x, \fA_{16}, \fc)
\end{equation}
This is because for any prime ideal $(\lambda)$ in the class of $\fA_{16}$, there are exactly four primes lie  in $(\lambda)$ but with exact one of them  primary, whence
\begin{equation}\label{k=1:nocAepsilon}
\#\fA_{16}=\#A_{16}
\end{equation}
From Dirichlet prime ideal theorem over $\bQ(i)$ ( see Proposition \ref{mainthm-explicitformula}), we get
\[\pi (x, \fA_{16}, \fc) \sim \frac{\#\fA_{16}}{\#I(\fc)/P_\fc}\cdot \Li(x)\]
Let $\phi(16)$ be the number of $\brBig{\bZ[i]/16\bZ[i]}^\times$, then from the exact sequence (\ref{k=1:exact-seq}) we have $\#I(\fc)/P_\fc=\frac{\phi(16)}4$, whence from (\ref{k=1:C1'}) and (\ref{k=1:pi'}) we obtain:
\[\#C_1'(x,\alpha_1,0)\sim \frac{2\#\fA_{16}}{\phi(16)}\cdot\Li(x)\]

Then according to the following Lemma \ref{lem-k=1-noAt}, we get
\[\#C_1'(x,\alpha_1,0)\sim \frac1{2^4}\cdot\Li(x)\]
Note   $\#C_1(x)\sim\Li(x)$ and the bijection (\ref{bij-Ck}), we finish the proof of Theorem \ref{mainthm2} in the case $k=1$.

\begin{lemma}\label{lem-k=1-noAt} The cardinality of $A_{16}$ is
${\phi({16})}/{2^{5}}$.
\end{lemma}
\begin{proof}
By the definition of $A_{16}$, the class is  primary selects the subgroup $G=\overline{1+(2+2i)\bZ[i]}$ of the group $\brBig{\bZ[i]/16\bZ[i]}^\times$ which is four times of $G$ in cardinality. To determine those elements in $G$ selected by the conditions (i),(ii) we introduce two characters $\chi_j$ on $G$ defined by:
\[\chi_1(g)=i^{\frac{Ng-1}4},\quad \chi_2(g)=\Leg2{g}_4\]

Then conditions (i),(ii) are equivalent to find those $g\in G$ such that
\begin{equation}\label{k=1:chi-i}
\chi_1(g)=i^{\frac{\alpha_1-1}4},\quad \chi_2(g)=(-1)^{\frac{\alpha_1-9}8}
\end{equation}

This reminds us to study the behavior of $\chi_i$ over $G$. We can easily deduce that $\chi_j^2(g)=\Leg2{Ng}$ with $j=1,2$. Whence $\chi_1^2=\chi^2_2$ and they are characters of order $4$, since  $\chi_j^2(-1+2i)=-1$. Moreover we have $\chi_1(-1+2i)=i, \chi_2(-1+2i)=i^{-1}$ by Lemma \ref{lem-2quartic}. Therefore $\chi_1\not=\chi_2$. So the  character subgroup $G'$ generated by $\chi_1,\chi_2$   has $8$ elements.

Now we show that $G'$ is the dual group of $G/G_1\cap G_2$  with $G_i$ the kernel of $\chi_i$. We suffice to prove  $\#G/G_1\cap G_2=8$. From group isomorphism theorem, we only need to study $G/G_1$ and $G_1/G_1\cap G_2$. The first group $ G/G_1\simeq \mu_4$ as $\chi_1$ has order $4$, for the latter group  we have $\chi_2|_{G_1}: G_1/G_1\cap G_2\to \mu_4$, where $\mu_4$ is the   group of units of order $4$. Due to $\chi_1\not=\chi_2$ we know $\chi_2|_{G_1}$ is non-trivial. From $\chi_1^2=\chi_2^2$ we obtain
\[\chi_2(g_1)^2=\chi_1(g_1^2)=1, \quad g_1\in G_1\]
Whence $\chi_2|_{G_1}$ has order $2$ and $\#G_1/G_1\cap G_2=2$. Thus we get $\#G/G_1\cap G_2=8$. Therefore $G'$ is  the dual group of $G/G_1\cap G_2$ by counting cardinality.

Since $G'$ and $G/G_1\cap G_2$ are  dual groups, we can easily derive: there is a $g\in G$ with $\chi_j(g)=i^{x_j}$ for $j=1,2$  if and only if $i^{2x_1}=i^{2x_2}$ since $\chi_1^2=\chi_2^2$.  Note that $i^{\frac{\alpha_1-1}4}, (-1)^{\frac{\alpha_1-9}8}$ obviously satisfies this. Therefore there is a $g_0\in G$ such that (\ref{k=1:chi-i}) holds. Moreover by transition of $g_0$, we know all $g\in G$ satisfying (\ref{k=1:chi-i}) consists the subset $g_0(G_1\cap G_2)$, which selects an eighth of $G$.
\[\#A_{16}=\frac{\phi({16})}{2^{5}}\]
This completes the proof of the lemma.
\end{proof}

\subsection{The case $k \geqslant 2$}\quad

In this subsection we will use the method of Cremona-Odoni to prove Theorem \ref{mainthm2} with $k\ge 2$.

To define the similar map $\vphi$ as in \S2, we first define $T(x)$ to be all $n=p_1\cdots p_{k-1}\le x$ with $p_j$ arranged in ascending order such that
$$ p_l\equiv \alpha_l \pmod {16}, \;\Leg{p_l}{p_j}=(-1)^{B_{lj}}, 1\le l<j\le k-1 $$
From  independence of Legendre symbol property of Rhoades \cite{rhoades20092}, we have
\begin{equation}\label{eq:number-Tx}
\# T(x)\sim 2^{-\binom k2-2k+2} \cdot \# C_{k-1}(x)
\end{equation}
Similarly as $C_k'(x,\alpha,B)$ we define $T'(x)$ to be all $\eta=\lambda_1\cdots \lambda_{k-1}$ with $N\eta\le x$ and $\lambda_j\in\cP$ such that
$$ N\lambda_j\equiv \alpha_j \pmod {16}, \;\Leg{N\lambda_l}{N\lambda_j}=(-1)^{B_{lj}}, 1\le l<j\le k-1 $$ where we have arranged $N\lambda_l$   increasingly. Then we also have a bijection
\begin{equation}\label{bij-Tx}
T'(x)\longrightarrow T(x), \;\; \eta\mapsto N\eta
\end{equation}

Now we can prove Theorem \ref{mainthm2}:
\begin{proof}

Let $\tdeta\in\cP$ be the prime divisor of $\eta$ with maximal norm, then we define the map
\[\vphi: C_k'(x,\alpha, B)\to T'(x),\;\; \eta\mapsto  \eta/\tdeta \]
parallel as in \S2.  Now we divide into two cases according to $z_k=0$ or $1$.\\

For the case $z_k=0$: we know an $\epsilon=\prod_1^{k-1}\lambda_j \in T'(x)$ lies in the image of $\vphi$ if and only if there is a prime $\lambda\in\cP$ with $N\tilde\epsilon<N\lambda\le x/N\epsilon$ such that
\begin{itemize}
\item[(i):] $N\lambda\equiv \alpha_{k}  \pmod {16}$ and $\Leg {N\lambda}{N\lambda_j}=(-1)^{B_{jk}}$ with $1\le j\le k-1$;
\item[(ii):] $\brBig{\frac{2}{\lambda}}_4\Leg\lambda{\theta_1}=
\brBig{\frac2\epsilon}_4\Leg{\epsilon/\theta_1}{\theta_1}(-1)^{\frac{\prod_1^k\alpha_j -1}{8}+\frac{\prod_1^k\alpha_j^{z_j}-5}4}$.
\end{itemize}
here $\tilde\epsilon\in \cP$ is the  prime divisor of $\epsilon$ with maximal norm.

Thus from Lemma \ref{lem-2quartic}, there is a unique subset $A_\epsilon$ of invertible primary residue  classes modulo ${16\epsilon}$ such that for a prime $\lambda$: the integer $ \lambda\epsilon $ belongs to  $C_k'(x,\alpha,B)$ if and only if $\lambda$ lies in $\cP$ and   $ A_\epsilon \pmod {{16\epsilon}}$ with norm in $\left(N\tdep, \; x/N\epsilon\right]$. Whence we obtain
\begin{equation}\label{eq:indep-induction-k2=1}
\# C_k'(x,\alpha,B)=\sum_{\epsilon\in T'(x) } g(\epsilon)
\end{equation}
with
\[g(\epsilon)=\#\BrBig{ \lambda \text{ prime of } \bZ[i] \;\Big|\; \lambda\in \cP, \; \lambda\in A_\epsilon  (\mod {{16\epsilon}}),\; N\tdep<N\lambda\le x/N\epsilon  }\]
As $\#A_\epsilon$ is an important part in main term, we list in the following lemma with proof postponed in the end of this section.
\begin{lemma}\label{lem-noAt}Let $\phi({16\epsilon})$ be the number of
$\brBig{ \bZ[i]/{16\epsilon} \bZ[i]}^\times$, then
 \[\#A_\epsilon=\frac{\phi({16\epsilon})}{2^{k+4}}\]
\end{lemma}

For simplicity, we introduce the notation  $\uset {N\eta\in A}{\oset{*}{\sum}} f(\eta)$ to denote $\sum_{\eta \in T(\infty), N\eta\in A} f(\eta)$ when $A\subset \bN$ as Cremona-Odoni.  Similarly as Cremona-Odoni, a Lemma parallel to Lemma \ref{lem-Cremona} also holds if we use $T'$ to substitute $C_{k-1}$, and  $\mu, \nu$ are defined the same as Lemma \ref{lem-Cremona}.\\

%As indicated in the introduction, to estimate $\# C_k(x,\alpha,B)$ , we have to study the
%behavior of $\uset {m<t\le n}{\oset{*}{\sum}}  \Li(xt^{-1})$ via Abel summation on the
%three sub-intervals of $[20,x^{\frac k{k+1}}]$ divided by $\mu, \nu$, while this has been
% studied in Cremona-Odoni, we cite it in the following:

Now we can estimate $\#C_k'(x,\alpha,B)$ in equation (\ref{eq:indep-induction-k2=1}):\\

First for $\epsilon\in T(x)$ with $N\epsilon\le 20$, then :
$$ g(\epsilon)\le  \pi(x/N\epsilon)$$ since every prime ideal corresponds to exact $1$ primitive prime elements, where  \[\pi(y)=\sum_{N\fp\le y} 1\sim \Li(y)\]
by prime ideal Theorem over $\bQ(i)$.  So all these $\epsilon$ with $N\epsilon\le 20$ contribute at most $O(\frac x{\log x})$.

Second for $N\epsilon$ lies in $(20, \;\mu]$: similarly we have
$$g(\epsilon)=O\brBig{\Li(x/N\epsilon)}$$
Hence these $\epsilon$  contributes to
\[\sum_{20<t\le\mu} O\brBig{\Li(xt^{-1})}=O\brBig{\sum_{20<t\le \mu} \Li(xt^{-1})}=o\brBig{\frac x{\log x} \cdot (\log\log x)^{k-1}   }\]
by  Lemma \ref{lem-Cremona}.

Similarly for $N\epsilon$ belonging to $(\nu, x^{\frac {k-1}{k}}]$, they also contribute to $o\brBig{\frac x{\log x} \cdot(\log\log x)^{k-1}  }$. While for those $N\epsilon> x^{\frac {k-1}{k}}$, they have no contribution:  as in this case we have $N\tilde\epsilon> x^{\frac1k}$, but this contradicts that $N\tilde\epsilon<N\lambda\le x/N\epsilon<x^{\frac1k}$.

Consequently  we obtain:
\begin{equation*}%\label{eq:number of Ck1k2-x-alpha-B}
\#C_{k}'(x,\alpha,B)\sim \frac 12 \sum_{\mu<N\epsilon\le \nu}^* \brBig{\pi'(x/N\epsilon, A_\epsilon, {16\epsilon} )-\pi'(N\tdep, A_\epsilon,{16\epsilon} )}
%+o\brBig{\frac x{\log x}\cdot (\log\log x)^{k-1}  }
\end{equation*}
%where $\pi'(y,A,q)$ is the number of primes $\lambda$ in $\bZ[i]$ with $N\lambda\le y$ and
%$\lambda\in A \pmod q$,  and the factor $\half$ comes from the imaginary part of $\lambda$ is positive.
Recall that $\pi'(y,A,\gamma)$ is defined under (\ref{k=1:C1'}). For the contribution of above latter terms we yield:
\[\sum_{\mu<N\epsilon\le\nu}^* \pi'(N\tdep,A_\epsilon,{16\epsilon} )\le \nu\cdot O\brBig{\frac \nu{\log \nu}}=O\brBig{\frac{\nu^2}{\log \nu}}=o\brBig{\frac x{\log x} \cdot (\log\log x)^{k-1}  }\] Therefore   we have
\begin{equation}\label{eq:CkxalphaB}
\#C_{k}'(x,\alpha,B)\sim \frac 12 \sum_{\mu<N\epsilon\le \nu}^*
\pi'(x/N\epsilon, A_\epsilon, {16\epsilon} )
%+o\brBig{\frac x{\log x}\cdot (\log\log x)^{k-1}  }
\end{equation}

%Thus the estimation of $\#C_{k}'(x,\alpha, B)$ is reduced to estimate
%\begin{equation}\label{eq:est-prime}
%{\sum_{\mu<N\epsilon\le \nu}^* \pi'(x/N\epsilon, A_\epsilon,{16\epsilon} )}
%\end{equation}

Similar as the case $k=1$, we define $\fc=\fc_\epsilon$ to be the ideal generated by ${16\epsilon}$, then  we have the following long exact sequence:
\begin{eqnarray}\label{exact-seq}
\xymatrix{
1\ar[r]^{} & \bZ[i]^\times \ar[r]^{} &
\brBig{\bZ[i]/\fc}^\times \ar[r]^{f} &
I(\fc)/P_\fc\ar[r]^{} & 1   }
\end{eqnarray}
%where $f$ maps every $\fc$-invertible Gauss integer $a$ to the class generated by $(a)$.
%In fact, the exactness of (\ref{exact-seq}) can be verified directly using $\bZ[i]$ having class number $1$.
If we denote $\fA_\epsilon=f(A_\epsilon)$, % and $\pi'(y,\fA_\epsilon,\fc)$ denote
%all the prime ideals lie  in the class of $\fA_\epsilon$ with norm less or equal to $y$.
then similarly as (\ref{k=1:pi'}) and (\ref{k=1:nocAepsilon}) we get
\[\pi'(x,A_\epsilon, {16\epsilon})=\pi (x, \fA_\epsilon, \fc)\]
%This is because for any prime ideal $(\lambda)$ in the class of $\fA_\epsilon$, there are
%exactly four primes lie  in $(\lambda)$ but with exact one of them  primary, whence
and
\begin{equation}\label{eq:nocAepsilon}
\#\fA_\epsilon=\#A_\epsilon
\end{equation} with $\pi(x,\fA_\epsilon,\fc)$ defined under (\ref{k=1:exact-seq}). 
Hence by equation (\ref{eq:CkxalphaB}) we reduce to estimate:
\begin{equation*}%\label{eq:est-prime ideal}
{\sum_{\mu<N\epsilon\le \nu}^* \pi (x/N\epsilon, \fA_\epsilon, \fc )}
\end{equation*}

Then by the standard relation of $\pi(y,\fA,\fc)$ and $\psi (y,\fA,\fc)$, we only need to  estimate
\begin{equation}\label{eq:ind-psi(xt^{-1})}
{\sum_{\mu<N\epsilon\le \nu}^* \psi(x/N\epsilon, \fA_\epsilon,\fc)}
\end{equation}
where
\[\psi(y,\fA,\fc):=\sum_{N\fa\le y\atop \fa\in \fA \mod P_\fc} \Lambda(\fa)\]

By orthogonality of characters and above exact sequence (\ref{exact-seq}), we have
\begin{equation*}% \label{eq:exp of psi-x,At,8epsilon}
\psi(y,\fA_\epsilon,\fc)=\frac4{\phi({16\epsilon})}  \sum_{\chi \mod \fc_\epsilon} \psi(y,\chi) \sum_{[\fa]\in \fA_\epsilon} \overline{\chi(\fa)}
\end{equation*}
where $\chi$ runs over all the characters of $I(\fc)/P_\fc $,
% and $\phi({16\epsilon})$ is the cardinality of $\brBig{\bZ[i] / \fc }^\times$
and \[\psi(y,\chi):=\sum_{N\fa\le y} \chi(\fa) \Lambda(\fa)\]
Consequently we divide the sum (\ref{eq:ind-psi(xt^{-1})}) into    three parts according to principal characters, non-principal characters of modulus multiple of $\dag_1$ or not:
% $$
\begin{eqnarray*}
\sum_{\mu<N\epsilon\le \nu}^* \psi(x/N\epsilon, \fA_\epsilon,\fc)&=&(I)+(II)+(III)\\
(I)&=& \sum_{\mu<N\epsilon\le \nu}^*\frac{4}{\phi({16\epsilon})}\cdot \#\fA_\epsilon\cdot \psi(x/N\epsilon, \chi_0) \\
(II)&=& \sum_{\mu<N\epsilon\le \nu,\atop \dag_1|\fc_\epsilon}^*\frac{4}{\phi({16\epsilon})}\sum_{\chi \mod\fc_\epsilon}'
\psi(x/N\epsilon,\chi)\sum_{[\fa]\in\fA_\epsilon} \overline{\chi(\fa)} \\
(III)&=&\sum_{\mu<N\epsilon\le \nu,\atop \dag_1\nmid\fc_\epsilon}^*\frac{4}{\phi({16\epsilon})}\sum_{\chi\mod \fc_\epsilon}'
\psi(x/N\epsilon,\chi)\sum_{[\fa]\in\fA_\epsilon} \overline{\chi(\fa)}
\end{eqnarray*}
where $\dag_1$ is the conductor of the exceptional primitive character in (2) of Proposition \ref{thm-Fogels} with $z=16^2\nu$, and  $\oset'\sum$ denotes all non-principal characters of a fixed modulus. These sums are estimated in the following lemma
\begin{lemma} \label{lem-3sum}
\begin{eqnarray*}
(I)&\sim & \frac1{(k-1)\cdot2^{k+2}}
\cdot \#T'(x)\cdot  \log x \cdot \log\log x \\
(II)&= & O\brBig{x\log^{-99}\nu } \\
(III)&=& o\brBig{\frac x{\log x}}
\end{eqnarray*}
\end{lemma}

We postpone the proof of this Lemma. From this Lemma and the bijection (\ref{bij-Tx}) we arrive at:
\begin{eqnarray*}
\sum_{\mu<N\epsilon\le \nu}^* \psi(x/N\epsilon, \fA_\epsilon,\fc)\sim
\frac1{(k-1)\cdot2^{k+2}}
\cdot \#T(x)\cdot  \log x \cdot \log\log x
\end{eqnarray*}
Whence from the equations (\ref{eq:number-Tx}) and (\ref{eq:CkxalphaB}) we have
\begin{eqnarray*}
\#C_{k}'(x,\alpha,B)&\sim&
\frac1{(k-1)\cdot 2^{k+3}}\cdot\#T(x)\cdot\log\log x\\
&\sim& \frac1{(k-1)\cdot2^{\binom k2+3k+1}}
\cdot\log\log x\cdot \#C_{k-1}(x)\\
&\sim&\frac1{2^{\binom k2+3k+1}}\cdot \#C_{k}(x)
\end{eqnarray*}

For the case $z_k=1$ we can prove similarly. Thus we complete the proof of Theorem \ref{mainthm2} by noting the bijection (\ref{bij-Ck}).
\end{proof}

Now we  prove Lemma \ref{lem-3sum}:
\begin{proof}

The first sum (I): by equation (\ref{eq:nocAepsilon}) and Lemma \ref{lem-noAt}, we have  $\#\fA_\epsilon=\frac{\phi({16\epsilon})}{2^{k+4}}$.
Then  Lemma \ref{lem-Cremona} implies
\begin{eqnarray*}\label{sum-1}
(I)&=&\frac1{2^{k+2}} \sum_{\mu<N\epsilon\le \nu}^* \psi(x/N\epsilon)=\frac{1+o(1)}{2^{k+2}}\sum_{\mu<N\epsilon\le \nu}^* \log( x/N\epsilon)\Li (x/N\epsilon) \\
&=&\frac{1+o(1)}{2^{k+2}}\cdot \log x\sum_{\mu<N\epsilon\le \nu}^* \Li(x/N\epsilon)   \\
&\sim &   \frac1{(k-1)\cdot2^{k+2}}
\cdot \#T'(x)\cdot  \log x \cdot \log\log x
\end{eqnarray*}

For the second sum (II):   the trivial estimation gives:
\[(II) \ll   \sum_{\mu<N\epsilon\le \nu,\atop \dag_1|\fc_\epsilon}^*
\psi(x/N\epsilon)
\ll  x\sum_{\mu<N\epsilon\le \nu,\atop \dag_1|\fc_\epsilon}^* (N\epsilon)^{-1} \]
If we denote $\dag_1=\brBig{(1+i)^{e}\beta_1\cdots\beta_t}$ with $e\le 8, t\le k-1$ and $\beta_j\in \cP$, then for any $s$ with $\mu<s N\dag_1\le\nu$ there are exactly $2^{k-1-t}$ ideals $\fc'$ above $s$ with $\dag_1\fc'=\fc_\epsilon$ for some $\epsilon\in T'(\infty)$. Therefore
\begin{equation}\label{ineq:sum 2}
(II)\ll   x N\dag_1^{-1}  \sum_{\mu<sN\dag_1\le \nu}^* s^{-1} \le x N\dag_1^{-1}\cdot \log \nu
\end{equation}
But we  we may assume that
\[N\dag_1>(\log \nu)^{100}\]
As from Page Theorem part of Proposition\ref{thm-Fogels}, we chose $z=16^2\nu$, then if the Siegel zero exists,  we have the Siegel zero $\beta$ of modulus $\dag_1$ satisfying
\[\beta>1-\frac{c_0}{\log (16^2\nu)}\]
Then Siegel zero part of Proposition  \ref{thm-Fogels} implies that for any $\epsilon>0$, there is a $c(\epsilon,2)$ such that
\[\beta\le 1-c(\epsilon,2)D^{-\epsilon}\]where $D=4N\dag_1$.
Thus if we chose $\epsilon=200$, then $N\dag_1>\log^{100}\nu$.

Hence (\ref{ineq:sum 2}) implies
\[(II)\ll  x\log^{-99}\nu \]

For the third sum (III): there is no Siegel zero, so from the explicit formula (\ref{eq:explicit-we need}) we have for any $\psi(x/N\epsilon,\chi)$ in sum (III), there exist a positive constant $c$ such that
\[\psi(x/N\epsilon,\chi)\ll  
\frac x{N\epsilon}\cdot\log^2x\cdot\exp\brBig{-\frac{c\log (x/N\epsilon)}{\log {N\epsilon}}}+ \frac x{N\epsilon^{ 5}}\log ^2x+x^{\frac14}{N\epsilon}^{-\frac14}\log (x/{N\epsilon}) \]
where we have chosen $T={N\epsilon}^4$ and used  $N\dag\le 16^2{N\epsilon}$. Corresponding to these three terms, we arrive at
\[(III)=\Sigma_1+\Sigma_2+\Sigma_3\]
\begin{eqnarray*}
\Sigma_1&=&x\log^2x \sum_{\mu<N\epsilon\le \nu,\atop \dag_1\not|\fc_\epsilon}^*  \frac1{N\epsilon}\cdot \exp\brBig{-c \frac{\log x/N\epsilon}{\log {N\epsilon} }}\\
&\ll  & x\log^2x \cdot\exp\brBig{-c'(\log\log x)^{100} }\cdot
\sum_{\mu<N\epsilon\le \nu,\atop \dag_1\not|\fc_\epsilon}^*  \frac1{N\epsilon}   \\
&\ll  &x\log^3x \cdot\exp\brBig{-c'(\log\log x)^{100}  }
\end{eqnarray*}
% For $\Sigma_2$ and $\Sigma_3$ we have:
\begin{eqnarray*}
\Sigma_2&=& x\log^2x  \sum_{\mu<N\epsilon\le \nu,\atop \dag_1\not|\fc_\epsilon}^*  {N\epsilon}^{-5}\ll  x\log^2x\cdot \mu^{-4}\ll  x\log^{-200}x \\
\Sigma_3&\ll  & x^{\frac14}\log x \sum_{\mu<N\epsilon\le \nu,\atop \dag_1\not|\fc_\epsilon}^* {N\epsilon}^{-\frac14}\ll  
x^{\frac14}\log x\cdot \nu^{\frac34}\ll  x^{\frac12}
\end{eqnarray*}
Hence we arrive at
\[(III)=o\brBig{\frac x{\log x}}\]
\end{proof}

Now we prove  Lemma \ref{lem-noAt}:
\begin{proof}
According to (i),(ii) and Lemma \ref{lem-2quartic}, we know that $A_\epsilon$ represents those primary class $ a\pmod {{16\epsilon}}$ such that
\begin{itemize}
\item[(i')] $N a\equiv \alpha_{k}  \pmod {16}$, and $\Leg {Na}{N\lambda_j}=(-1)^{B_{jk}}$ for $1\le j\le k-1$;
\item[(ii')] $\brBig{\frac {2}{ a}}_4\Leg a{\theta_1}=
\brBig{\frac2{\epsilon}}_4\Leg{\epsilon/\theta_1}{\theta_1} (-1)^{\frac{\prod_1^k\alpha_j -1}{8}+\frac{\prod_1^k\alpha_j^{z_j}-5}4}$.
\end{itemize}

From Chinese Remainder Theorem, we have the following identification:
\[ \brBig{\bZ[i] / {16\epsilon}\bZ[i] }^\times\simeq \brBig{\bZ[i]/16\bZ[i]}^\times \times \prod_{j=1}^{k-1} \brBig{\bZ[i]/\lambda_j\bZ[i]}^\times \]
given by $a\mapsto (a_0, a_1,\cdots, a_{k-1})$ with $a_j$ the corresponding image of $a$ modulo $\lambda_j$ and $\lambda_0:=16$. Then the residue symbol $\Leg{\cdot}{\theta_1}$ is trivial on those component $\lambda_j\nmid\theta_1$, and similarly for other residue symbols.  Hence the condition $\Leg {Na_j}{N\lambda_j}=1$ selects half of the $\brBig{\bZ[i]/\lambda_j\bZ[i]}^\times $-part, since the norm map induces an isomorphism 
$\brBig{\bZ[i]/\lambda_j\bZ[i]}^\times \simeq \brBig{\bZ/p_j\bZ}^\times$ as $p_j=N\lambda_j$ splits completely in $\bZ[i]$.

For the part of $\brBig{\bZ[i]/16\bZ[i]}^\times$, we use the same notation as in Lemma \ref{lem-k=1-noAt}. With $a_1,\cdots,a_{k-1}$ chosen such that $\Leg{Na_j}{N\lambda_j}=1$,  then the remaining conditions  of (i'),(ii') are equivalent to
\begin{equation}\label{eq:chi12-kernel}
\chi_1(g)=i^{\frac{\alpha_k-1}4}, \quad \chi_2(g)=i^\delta
\end{equation}
where $g\in G$ and $i^\delta=\brBig{\frac2\epsilon}_4
\Leg{\epsilon/\theta_1}{\theta_1} \cdot \prod_{\lambda_j|\theta_1}\Leg{a_j}{\lambda_j}\cdot(-1)^{\frac{\prod_1^k\alpha_j -1}{8}+\frac{\prod_1^k\alpha_j^{z_j}-5}4}$.

Similar as Lemma \ref{lem-k=1-noAt}, the existence of $g\in G$ such that (\ref{eq:chi12-kernel}) holds is equivalent to $i^{2\cdot \frac{\alpha_k-1}4}=i^{2\delta}$. Now we verify this:
\[i^{2\delta}=\Leg{2^2}\eta_4=\Leg2{N\eta}=\Leg2{\alpha_1\cdots\alpha_{k-1}}=\Leg2{\alpha_k}\]
as $\prod_{j=1}^k\alpha_j\equiv 1 \pmod 8$, while
$$i^{2\cdot\frac{\alpha_k-1}4}=(-1)^{ \frac{\alpha_k-1}4}=\Leg2{\alpha_k}=i^{2\delta}$$
Therefore the condition (\ref{eq:chi12-kernel}) selects  an eighth of $G$ similarly. Consequently
\[\#A_\epsilon=\frac{\phi({16\epsilon})}{2^{k+4}}\]
This completes the proof of the lemma.
\end{proof}

\section{Distribution of Congruent Elliptic Curves}

In this section, we will prove the main Theorem and the distribution of those congruent elliptic curves with rank $0$ and $2$-primary part of Shafarevich-Tate group isomorphic to $\brbig{\bZ/2\bZ}^4$. 

According to the strategy explained in the introduction, we first interpret some conceptions related to $8$-rank in Gauss genus theory. We refer to \S3 of our previous paper \cite{wzj2015congruent}. Let $n=p_1\cdots p_k\equiv 1 \pmod8$ in $Q_k$, denote $\cA=\cA_n$ the ideal class group of $\bQ(\sqrt{-n})$. Then the $2^j-$rank $h_{2^j}(n)$of $\cA$ is defined to be $\rank_{\bF_2}2^{j-1}\cA/2^j\cA$ with the multiplications in $\cA$ written additively, hence  $2^j\cA$ denotes the subgroup consisting of  $2^j$-power of elements in $\cA$.
Note that from definition we can easily get $h_4(n)=\rank_{\bF_2} \cA[2]\cap 2\cA$, where $\cA[2] $ denotes the subgroup consisting of elements with square trivial.
Then Gauss genus theory implies that there is a $2$ to $1$ epimorphism
\begin{equation}\label{epimor}
\theta: \Brbig{ X\in \bF_2^{k+1} \big| RX=0 }\longrightarrow \cA[2]\cap 2\cA
\end{equation}with $\theta(X_0)$ trivial,
where $X_0=(1,\cdots,1,0)^T$ and $R$ is a $k\times(k+1)$ matrix over $\bF_2$ defined by $$R=\brbig{A \big| \fb}$$
with $A=(a_{ij})_{k\times k}$ and $\fb=\brBig{\ALeg2{p_1},\cdots,\ALeg2{p_k}}^T$, here
$a_{ii}=\sum_{l\not=i} a_{il}, a_{ij}=\ALeg{p_j}{p_i}$ with $i\not=j$.
Whence we have $h_4(n)=k-\rank R$. 

To count the number of certain congruent  elliptic curves with $n\in Q_k$ conveniently,   we assume  that: 
\begin{center}
\emph{ The Redie matrix $R_n$ is defined with the $p_j$  arranged increasingly}
\end{center}
whence $R=R_n$ and $A=A_n$ is completely determined by $n$.

\subsection{Distribution of $\Sha(\En/\bQ)[2^\infty]\simeq \brbig{\bZ/2\bZ}^2$}\quad

Theorem 1 of \cite{wzj2015congruent} characterizes that $n\in P_k$ { if and only if  }
\begin{equation}\label{equivalence of Pk}
h_4(n)=1,\;h_8(n)\equiv \frac{d-5}4 \pmod 2
\end{equation}
where $d=\prod_{j=1}^k p_j^{x_j}$ with $X=(x_1,\cdots,x_k,x_{k+1})^T\not=X_0$ satisfying $RX=0$. % here $X_0=(1,\cdots,1,0)^T$. 
Moreover two choice of $X_1,X_2$ doesn't affect $\frac{d-1}4\pmod 2$.
%By (\ref{epimor}), there are two $X_1, X_2$ such that $RX_j=0$, hence two $d_1,d_2$
%correspondingly. Note that $X_0$ also lies in $\ker(R)$, which has dimension $2$ since
%$h_4(n)=1$ by the epimorphism (\ref{epimor}). Thus  $\ker(R)=\Brbig{0, X_0, X_1, X_2 }$,
%and $X_1+X_2=X_0$. Consequently $d_1d_2=n$, so they doesn't affect
In fact, if $h_4(n)=1$ there are two cases according to $\rank A$ as in the proof of Theorem 1 of \cite{wzj2015congruent}:
\begin{itemize}
\item[(i)] If $\rank A_n=k-1$, let $x=(x_1,\cdots,x_k)$ be a non-trivial solution of $Ax=\fb$, then $d=\prod_{j=1}^kp_j^{x_j}$. Then Theorem 3.3(iii),(iv) of Jung-Yue \cite{yueJung2011eightrank} implies that  $h_8(n)=1 $ if and only if 
\[\Leg{2d}{d'}_4\Leg{2d'}d_4=
    (-1)^{\frac{n-1}{8}}\] with $n=dd'$. Then by equation (\ref{eq:leg2dd'2d'd}) this is equivalent to
    \[\Leg{\theta_2}{\theta_1}\Leg2n_4=(-1)^{\frac{n-1}8}\]
    where $\theta_1, \theta_2$  is the primary integer lying above $d,d'$ with every prime factor in $\cP$.
\item[(ii)]    If $\rank A=k-2$, let $x=(x_1,\cdots,x_k)\not=0, (1,\cdots,1)^T$ such that $Ax=0$, then $d=\prod_{j=1}^k p_j^{x_j}$ and $d\equiv 5 \pmod8$. Then Theorem 3.3(ii) of Jung-Yue \cite{yueJung2011eightrank} implies that $h_8(n)=1$ if and only if
    $$\Leg{d}{d'}_4\Leg{d'}d_4=-1$$ where $d'd=n$.
\end{itemize}
\begin{rem}
We remark that there is a typo in Theorem 3.3 (iv) of Jung-Yue \cite{yueJung2011eightrank}, which is corrected in above (i).
\end{rem}

According to this result and our strategy,
%we have to  consider $k\times k$ symmetric-$\bF_2$ matrix with rank $k-1$ and $k-2$. Hence
we use $\cB=\cB_k$   to denote all the $k\times k$ symmetric-$\bF_2$ matrix  with rank $k-1$ and   every row sum $0$, similarly use $\cB'$ to denote those with rank $k-2$. Now we divide into two cases according to $B\in\cB, \cB'$ respectively:\\

(a): For any $B\in\cB=\cB_k$ and any $\alpha=(\alpha_1,\cdots, \alpha_k)$ with $\alpha_i\in \Brbig{ 1, 5, 9, 13 }$ and $\prod_{j=1}^k \alpha_j\equiv 1  \pmod 8$: then the contribution of those $n\le x$ with $A_n=B,  p_j\equiv \alpha_j\pmod{16}$ to $\#P_k(x)$ is the number of $C_k(x,\alpha,B)$, where $n=p_1\cdots p_k$ with 
$p_j$ arranged in ascending order. Therefore all of these $B, \alpha$ contributes to $\#P_k(x)$ is
\[\Sigma_1=\sum_{B\in\cB}\;\;\sum_{\alpha\atop  8|\alpha_1\cdots \alpha_k-1} \#C_k(x,\alpha, B)\]
By independence of residue symbol Theorem \ref{mainthm2}, we know this asymptotically equals to
\[2^{-1-\binom{k}2-3k}\cdot \#C_k(x) \cdot \sum_{B\in\cB}\;\sum_{\alpha \atop 8|\alpha_1\cdots \alpha_k-1}1=2^{-2-k-\binom{k}2}\cdot\#C_k(x)\cdot \#\cB\]
So we reduce
to compute $\#\cB$, which can be accomplished by a result in  %Rhoades \cite{rhoades20092},
Brown and many coauthors \cite{brown2006trivial}:
\begin{prop}[]\label{thm-brown et al} Let $\cB_{k,r}$ denote all the $k\times k$ symmetric matrix over $\bF_2$ of rank $r\le k$, then
\[\#\cB_{k,r}=2^{\binom{r+1}2}\cdot u_{r+1}\cdot\prod_{i=0}^{k-r-1}
\frac{2^k-2^i}{2^{k-r}-2^i}\]
where $u_r$ is defined in Theorem \ref{mainthm0}.
\end{prop}
Note every  $B\in \cB$ corresponds to a  $B'\in \cB_{k-1,k-1}$ by summating  all rows to last row then delete last row, and similarly for column. Whence we have
\begin{equation}\label{eq:no-cB-k}
\#\cB=\#\cB_k=u_k\cdot 2^{\binom{k}2}
\end{equation}
So we arrive at
\[\Sigma_1\sim 2^{-2-k}\cdot u_k\cdot \#C_k(x) \]\quad\\
% $\sim \frac{2^{-2-k}u_k}{ (k-1)!} \cdot  \frac{x (\log\log x)^{k-1} }{\log x}$

(b): For   $B\in \cB'$, we denote $\Sigma_B$ to be the set of    $\alpha=(\alpha_1,\cdots,\alpha_k)$ with $\alpha_i\in \Brbig{ 1, 5, 9, 13 }$ and $\prod_{j=1}^k\alpha_j \equiv 1 \pmod 8$ such that
\[\rank_{\bF_2} \brbig{B \big| \fb_\alpha}=k-1\]
where $\fb_\alpha=\brBig{\ALeg2{\alpha_1}, \cdots, \ALeg2{\alpha_k}}^T$. Then for $B\in \cB', \alpha\in\Sigma_B$, the contribution of  those $n=p_1\cdots p_k\le x$ with $p_1<\cdots <p_k, A_n=B$ and $ p_j\equiv \alpha_j\pmod{16}$ to $\#P_k(x)$ is the number of $C_k(x,\alpha,B)'$, which is defined to be all $n=p_1\cdots p_k\in C_k(x)$ with $p_j$ arranged increasingly   satisfying:
\begin{enumerate}
\item[(1)] $p_l\equiv\alpha_l \pmod{16}$;
\item[(2)] $\Leg{p_l}{p_j}=(-1)^{B_{lj}}, \; 1\le l<j\le k$;
\item[(3)] $\Leg d{d'}_4\Leg{d'}d_4=-1$.
\end{enumerate}
where $dd'=n$ and $d=\prod_{j=1}^k p_j^{x_j}$ with $x\not=0, (1,\cdots,1)^T$ a vector in $\bF_2^k$ such that $Ax=0$. Note that this case implies that  $k\ge 2$.

Before we count all the contribution of  $B\in \cB'$ and $\alpha\in \Sigma_B$, we first counting the number of $\cB'$ and $\Sigma_B$ respectively.

(b1): As every row sum of  $B\in\cB'$ is $0$, thus adding all rows to the last row then the last row is $0$ then delete the last row, similarly for the last column as $B$ is symmetric, thus we get a $B'\in \cB_{k-1, k-2}$, thus from Proposition \ref{thm-brown et al} we have
\begin{equation}\label{no-cB'}
\#\cB'=2^{\binom{k-1}2}\cdot u_{k-1}\cdot(2^{k-1}-1)
\end{equation}

(b2): For $B\in \cB'$, we want to count the number of those  $\bF_2$-matrix $\fb=\fb_{k\times 1}$ with column sum $0$ such that
\[\rank_{\bF_2}\brbig{B\big|\fb}=k-1\]
With similar elementary transforms  as $B$, then  $\brbig{B\big|\fb}$ corresponds to \[\brBig{B' \big| \fb' }\]
where $\fb'$ obtains from $\fb$ with no last term, then the rank of $\brBig{B' \big| \fb' }$ is $k-1$ means that $\fb'$ is not in the image of $B'$, thus there are exactly $2^{k-2}$ many $\fb'$. Whence there are $2^{k-2}$ such $\fb$. To every $\fb$ there are $2^k$ many $\alpha\in \Sigma_B$: as $b_j=\ALeg2{\alpha_j}$ determines $\alpha_j \pmod8$, then any $\alpha_j$ has exact two choice. Consequently we have
\begin{equation}\label{no-SigmaB'}
\#\Sigma_B=2^{2k-2}
\end{equation}

Similarly as Theorem \ref{mainthm2}, we have the following independence of residue symbol property for $\# C_{k}(x,\alpha,B)'$:
\begin{equation}\label{eq:indLeg-OurSituation-rk k-2 in 1 mod 4 all factor}
\# C_{k}(x,\alpha,B)' \sim 2^{-1-3k-\binom{k}2} \# C_k(x)
\end{equation}\quad
Thus   the contribution of all $B\in\cB'$ and $\alpha\in \Sigma_B$  is
\[\Sigma_2=\sum_{B\in \cB'}\;\sum_{\alpha\in \Sigma_B } \#C_k(x,\alpha,B)' \\
\sim 2^{-1-3k-\binom k2}\sum_{B\in \cB'}\;\sum_{\alpha\in \Sigma_B }\#C_k(x)\]
Then from (\ref{no-cB'}) and (\ref{no-SigmaB'}), $\Sigma_2$  asymptotically equals to
\[2^{-1-3k-\binom k2}\#\cB'\cdot\#\Sigma_B\cdot\#C_k(x)
=2^{-2-2k}(2^{k-1}-1)u_{k-1}\#C_k(x)\]

From (i),(ii) we know $\#P_k(x)=\Sigma_1+\Sigma_2$, therefore
\begin{eqnarray*}
\#P_k(x)&\sim&2^{-2-k}\brBig{u_k+ (2^{-1}-2^{-k})u_{k-1}} \#C_k(x)  \\
\lim_{x\to\infty}\frac{\#P_k(x)}{\#Q_k(x)}&=& \half \brBig{u_k+(2^{-1}-2^{-k})u_{k-1}}
\end{eqnarray*}
where we have used the fact that
\[\#Q_k(x)\sim 2^{-1-k}\cdot \#C_k(x)\]
which can be easily verified by independence of residue symbol property of Rhoades \cite{rhoades20092}. Whence we finish the proof of Theorem \ref{mainthm0}.

\subsection{Distribution of $\Sha(\En/\bQ)[2^\infty]\simeq \brbig{\bZ/2\bZ}^4$}\quad

In this subsection, we mainly discusss the distribution of congruent elliptic curves $\En$ with $\Sha(\En/\bQ)[2^\infty]\simeq \brbig{\bZ/2\bZ}^4$, see the following theorem:

\begin{thm}[]\label{mainthm3}
Let $\tdQ_k(x)$ be the set of positive squarefree integers $n\le x$ with exact $k$ prime factors and all prime factors of $n$ are congruent to $1\pmod8$, and $\tdP_k(x)$ consists of those  $n\in \tdQ_k(x)$ satisfying
	$$\rank_\bZ \En(\bQ)=0, \qquad \Sha(\En/\bQ)[2^\infty]\simeq \brbig{\bZ/2\bZ}^4$$ 
Then for any $k\ge2$ we have
	\[\lim_{x\to\infty}\frac{\#\tdP_k(x)}{\#\tdQ_k(x)}\ge
	 2^{k-4}\cdot \sum_{j_1+j_2=k\atop j_1,j_2\ge1} u_{j_1}u_{j_2}
	 \cdot {\binom{k}{j_1}}\cdot{2^{-j_1j_2}} \]
	where $u_j$ is defined in Theorem \ref{mainthm0}.
\end{thm} 

Given $k,k'\ge1$,   Theorem 1.2 and Remark 4 of our previous paper \cite{wzj2015congruent} shows that $n\in \tdP_{k+k'}$ if $n=dd'\in \tdQ_{k+k'}$ with $\omega(d)=k, \omega(d')=k'$ such that
\begin{itemize}
\item[(i)] $\Leg{p}{p'}=1$ with $p,p'$ any prime divisor of $d,d'$ respectively;
\item[(ii)] $h_4(d)=h_4(d')=1$;
\item[(iii)] $\Leg2{d}_4=(-1)^{\frac{d-9}8},\; \Leg2{d'}_4=(-1)^{\frac{d'-9}8}$ and $\Leg{d}{d'}_4=\Leg{d'}{d}_4=1$.
\end{itemize}
In fact Theorem 4.1 and Corollary 1 of that paper \cite{wzj2015congruent} give  more general conditions make $n$ lie in $\tdP_{k+k'}$, but for easy of notation we just limited to (i)-(iii).
% and we just enumerate the density under more general condition.

To count the contribution of those $n$ satisfying (i)-(iii), we need some notations. Let $k,k'\ge1$  and  $\sigma=\Brbig{\sigma_1,\cdots,\sigma_k}$ be an ascending subsequence of $1,\cdots,k+k'$ with exact $k$ elements, and $\sigma'=\Brbig{\sigma'_1,\cdots,\sigma_{k'}'}$ be the remained increasing subsequence of $1,\cdots,k+k'$ by deleting those elements of $\sigma$. Moreover we let $\cS$ to denote all these $\sigma$, then $\#\cS=\binom{k+k'}k$. We also let $\cR$ to denote the set of $\alpha=\brlr{\alpha_1,\cdots,\alpha_{k+k'}}$ with every $\alpha_j\in \Brbig{1,9}$, then $\#\cR=2^{2(k+k')}$.  For $\alpha\in\cR$, $B\in \cB_k, B'\in\cB_{k'}$ and $\sigma\in\cS$, we denote $C_{k,k'}(x,\alpha,B,B',\sigma)$ to be those $n=p_1\cdots p_{k+k'}\in\tdQ_{k+k'}(x)$ with $p_j$ arranged increasingly  such that
\begin{itemize}
	\item[(1)] $p_j\equiv \alpha_j\pmod{16}$ for $1\le j\le k+k'$;
	\item[(2)] $A_{d}=B, \;A_{d'}=B'$ with $d=\prod_{j=1}^k p_{\sigma_j},\; d'=\prod_{j=1}^{k'}p_{\sigma'_j}$;
	\item[(3)] $\Leg{p}{p'}=1$ with $p,p'$ any prime divisor of $d,d'$ respectively;
	\item[(4)] $\Leg2{d}_4=(-1)^{\frac{ \delta-9}8}, \;
	 \Leg2{d'}_4=(-1)^{\frac{ \delta'-9}8}$ with $\delta=\prod_{j=1}^k \alpha_{\sigma_j},
	 \delta'=\prod_{j=1}^{k'} \alpha_{\sigma_j'}$;
	 \item[(5)] $\Leg d{d'}_4=\Leg{d'}d_4=1$.
\end{itemize}
Then similar as independence of residue symbol Theorem \ref{mainthm2}, we have
\[\#C_{k,k'}(x,\alpha,B,B',\sigma)\sim \frac{1}{2^{4+3(k+k')+\binom{k+k'}{2}}}\cdot\#C_{k+k'}(x)\]
Moreover from (i)-(iii) we know $C_{k,k'}(x,\alpha,B,B',\sigma)$ is contained in $\tdP_{k+k'}(x)$. Therefore all the contribution of those $C_{k,k'}(x,\alpha,B,B',\sigma)$ with $\alpha\in \cR, \; B\in \cB_{k}, B'\in\cB_{k'}, \sigma\in \cS$ to $\tdP_{k+k'}(x)$ is
\begin{eqnarray*}
&&\sum_{\alpha\in\cR}\sum_{B\in\cB_k} \sum_{B'\in\cB_{k'}}\sum_{\sigma\in \cS}  \# C_{k,k'}(x,\alpha,B,B',\sigma)\\
	&\sim& \frac{\#\cR \cdot \#\cB_k\cdot\#\cB_{k'}\cdot \#\cS}{{2^{4+3(k+k')+\binom{k+k'}{2}}}}
	 \cdot\#C_{k+k'}(x)\\
	 &=&\frac{u_ku_{k'}\cdot \binom{k+k'}k\cdot 2^{2(k+k')+\binom k2+\binom{k'}2}}{{2^{4+3(k+k')+\binom{k+k'}{2}}}}
	 \cdot \#C_{k+k'}(x)\\
	 &=&\frac{u_ku_{k'}\cdot \binom{k+k'}k}{2^{4+k+k'+kk'}}\cdot\#C_{k+k'}(x)
\end{eqnarray*}
where we have used (\ref{eq:no-cB-k}). Note from independence of residue symbol property of Rhoades \cite{rhoades20092}, we have 
\[\#\tdQ_k(x)\sim 2^{-2k}\cdot\#C_k(x)\] 
Consequently we have
\[\lim_{x\to\infty}\frac{\#\tdP_k(x)}{\#\tdQ_k(x)}\ge
 2^{k-4}\cdot \sum_{j_1+j_2=k\atop j_1,j_2\ge1} u_{j_1}u_{j_2}
 \cdot {\binom{k}{j_1}}\cdot{2^{-j_1j_2}} \]

\textbf{Acknowledgements}\\

I am greatly indebted to Professor Ye Tian, my supervisor, for many  instructions and suggestions! I would like to thank Lvhao Yan for carefully reading the manuscript and giving valuable comments.

\bigskip

\noindent Zhangjie Wang, \\
 Yau  Mathematical Sciences Center, \\
Tsinghua University, Beijing 100084,China. \\
{\it {zjwang@math.tsinghua.edu.cn}  }

\end{CJK}
\end{document}